\documentclass[11pt]{amsart}
\usepackage{amsmath,amsthm,verbatim,amssymb,amsfonts,amscd,graphicx}
\usepackage[normalem]{ulem}
\usepackage{xcolor}
\usepackage{graphics}

\usepackage{euscript, enumerate}
\usepackage{url,hyperref}

\newcommand{\p}{\partial}

\newcommand{\e}{\varepsilon}

\newcommand{\eps}{\varepsilon}

\newcommand{\be}{\begin{equation}}
\newcommand{\ba}{\begin{aligned}}
\newcommand{\bee}{\begin{equation*}}
\newcommand{\ee}{\end{equation}}
\newcommand{\ea}{\end{aligned}}
\newcommand{\eee}{\end{equation*}}
\newcommand{\bea}{\begin{equation} \begin{aligned} }
\newcommand{\eea}{\end{aligned}\end{equation} }

\newcommand{\abs}[1]{\lvert#1\rvert}

\newcommand{\Lap}{\Delta}

\theoremstyle{plain}
\newtheorem{theorem}{Theorem}[section]

\newtheorem{corollary}[theorem]{Corollary}

\newtheorem{claim}{Claim}[section]

\theoremstyle{remark}

\theoremstyle{definition}
\newtheorem{definition}[theorem]{Definition}

\numberwithin{equation}{section}

\begin{document}
\title{Ricci limit flows and weak solutions}
\author{Beomjun Choi and Robert Haslhofer}

\begin{abstract}
In this paper we reconcile several different approaches to Ricci flow through singularities that have been proposed over the last few years by Kleiner-Lott, Haslhofer-Naber and Bamler. 
Specifically, we prove that every noncollapsed limit of Ricci flows, as provided by Bamler's precompactness theorem, as well as every singular Ricci flow from Kleiner-Lott, is a weak solution in the sense of Haslhofer-Naber. We also generalize all path-space estimates from Haslhofer-Naber to the setting of noncollapsed Ricci limit flows.\\
The key step to establish these results is a new hitting estimate for Brownian motion. A fundamental difficulty, in stark contrast to all prior hitting estimates in the literature, is the lack of lower heat kernel bounds under Ricci flow. To overcome this, we introduce a novel approach to hitting estimates that compensates for the lack of lower heat kernel bounds by making use of the heat kernel geometry of space-time.
\end{abstract}

\maketitle

\section{Introduction}

A family of Riemannian metrics $(g_t)_{t\in I}$, say on a closed $n$-dimensional manifold $M$, evolves by Ricci flow if
\begin{equation}
\partial_t g_t = -2\mathrm{Rc}(g_t).
\end{equation}

\bigskip

In a recent breakthrough \cite{bamler2020entropy,bamler2020compactness,bamler2020structure}, Bamler established a precompactness and partial regularity theory.  The limits provided by his precompactness theorem are so-called metric flows. A metric flow
 \begin{equation}
\mathcal{X}=\left(\mathcal{X},\mathfrak{t},(d_t)_{t\in I},(\nu_{x;s})_{x\in \mathcal{X}, s\in I,s\leq \mathfrak{t}(x)}\right),
\end{equation}
is given by a set $\mathcal{X}$, a time-function $\mathfrak{t}:\mathcal{X}\to \mathbb{R}$, complete separable metrics $d_t$ on the time-slices $\mathcal{X}_t=\mathfrak{t}^{-1}(t)$, and probability measures $\nu_{x;s}\in \mathcal{P}(\mathcal{X}_s)$ such that the Kolmogorov consistency condition and a certain sharp gradient estimate for the heat flow hold (see Section \ref{subsec_prelim_limit_flows} for details). In particular, any smooth Ricci flow can of course be viewed as metric flow by choosing $\mathcal{X}=M\times I$, defining $\mathfrak{t}$ as the projection on $I$, letting $d_t$ be the induced metrics on time slices, and setting $\nu_{x;s}$ to be the conjugate heat kernel measure based at $x=(p,t)$, i.e.
\begin{equation}
d\nu_{(p,t);s}(q) = K(p,t;q,s)\, d\mathrm{Vol}_{g_s}(q),
\end{equation}
where $K(p,t;q,s)$ is the heat kernel of the Ricci flow (specifically, $K$ solves the forwards heat equation as a function of $(p,t)$ and the conjugate heat equation as a function of $(q,s)$).\\
Under the noncollapsing assumption that the Nash entropy is bounded below, which is of course perfectly natural in light of Perelman's monotonicity formula \cite{Per1}, Bamler proved that the singular set $\mathcal{S}\subset\mathcal{X}$ of the limit flow has parabolic $\ast$-Minkowski dimension at most $n-2$.

\bigskip

In a different direction, a notion of weak solutions for the Ricci flow has been proposed a few years earlier by Naber and the second author. Specifically, it has been shown in \cite{HaslhoferNaber} that a smooth family of Riemannian metrics $(g_t)_{t\in I}$ evolves by Ricci flow if and only if the sharp infinite dimensional gradient estimate
\begin{equation}\label{char_sol}
|\nabla_p\mathbb{E}_{(p,t)}[F]| \leq \mathbb{E}_{(p,t)}[|\nabla^\parallel F|] 
\end{equation}
holds for all cylinder functions $F$ on the path space of its space-time. Here, $\mathbb{E}_{(p,t)}$ denotes the expectation with respect to the Wiener measure of Brownian motion starting at $(p,t)$, and $\nabla^\parallel$ denotes the parallel gradient, which is defined via a suitable stochastic parallel transport. Based on this characterization it has been proposed that a possibly singular space equipped with a time-function and a linear heat flow should be called a weak solution of the Ricci flow if and only if the sharp infinite dimensional gradient estimate on path space holds for almost every point $(p,t)$.
\bigskip

The goal of the present paper is to reconcile these two approaches. As we will explain in detail in Section \ref{sec_BM_stoch}, any noncollapsed Ricci limit flow $\mathcal{X}$ can be canonically equipped with a notion of Brownian motion and stochastic parallel transport. For now, let us just mention that Brownian motion $X_\tau$ starting at $x\in \mathcal{X}$ is simply characterized by the formula
\begin{equation}\label{def_Brownian_intro}
\mathbb{P}_x[ X_{\tau_1} \in B_1,\ldots,X_{\tau_k} \in B_k] 
=  \int_{B_1\times \ldots\times B_k}d\nu_{x;\mathfrak{t}(x)-\tau_1}(x_1)\ldots d\nu_{x_{k-1};\mathfrak{t}(x)-\tau_k}(x_k).
\end{equation}
Using these notions, we can now state our main theorem:
\begin{theorem}[Ricci limit flows and weak solutions]\label{thm_main_intro}
Given any noncollapsed Ricci limit flow $\mathcal{X}$, for any regular point $x=(p,t)$ we have the infinite dimensional gradient estimate
\begin{equation}\label{char_sol_thm}
|\nabla_p\mathbb{E}_{(p,t)}[F]| \leq \mathbb{E}_{(p,t)}[|\nabla^\parallel F|] 
\end{equation}
for all cylinder functions $F$. In particular, any noncollapsed limit of Ricci flows, as provided by Bamler's precompactness theorem, is a weak solution of the Ricci flow in the sense of Haslhofer-Naber.
\end{theorem}

In fact, our argument applies to any noncollapsed metric flow that satisfies the partial regularity properties from \cite{bamler2020structure} and solves the Ricci flow equation on its regular part.
In particular, viewing any singular Ricci flow from Kleiner-Lott \cite{KL1} as a metric flow as in \cite[Section 3.7]{bamler2020compactness}, we can confirm a prediction from \cite{HaslhoferNaber}:

\begin{corollary}[singular Ricci flows and weak solutions]
Every singular Ricci flow in the sense of Kleiner-Lott is a weak solution of the Ricci flow in the sense of Haslhofer-Naber.
\end{corollary}

As another important consequence of Theorem \ref{thm_main_intro} (Ricci limit flows and weak solutions) all other path-space estimates for smooth flows from \cite{HaslhoferNaber} generalize to the path-space of noncollapsed Ricci limit flows as well:
\begin{corollary}[estimates on path-space of Ricci limit flows]
The following estimates hold on path-space of any noncollapsed Ricci limit flow $\mathcal{X}$:
\begin{itemize}
\item For every cylinder function $F$ the induced martingale $F_\tau$ for almost every $(p,t)\in\mathcal{X}$ satisfies the quadratic variation estimate
\begin{equation}
 \mathbb{E}_{(p,t)} \left[\frac{d[F_\bullet]_\tau}{d\tau} \right] \leq 2 \mathbb{E}_{(p,t)} \left[\abs{\nabla^\parallel_\tau F}^2\right].
\end{equation}
\item For almost every $(p,t)\in\mathcal{X}$ the  Ornstein-Uhlenbeck operator on path-space, $\mathcal{L}_{\tau_1,\tau_2}=\int_{\tau_1}^{\tau_2} \nabla^{\parallel\ast}_\tau \, \nabla^\parallel_\tau d\tau$, satisfies the log-Sobolev inequality
\begin{multline}
 \mathbb{E}_{(p,t)} \left[  \left((F^2)_{\tau_2} \log\, (F^2)_{\tau_2} - (F^2)_{\tau_1} \log\, (F^2)_{\tau_1} \right)\right]\\
 \leq 4 \mathbb{E}_{(p,t)} \left[ \langle F,\mathcal{L}_{\tau_1,\tau_2} F\rangle \right].
\end{multline}
\item For almost every $(p,t)\in\mathcal{X}$ the  Ornstein-Uhlenbeck operator on path-space satisfies the spectral gap estimate
\begin{equation}
 \mathbb{E}_{(p,t)} \left[  (F_{\tau_2}-F_{\tau_1})^2 \right]\leq 2  \mathbb{E}_{(p,t)} \left[ \langle F,\mathcal{L}_{\tau_1,\tau_2} F\rangle \right].
\end{equation}
\end{itemize}
\end{corollary}
Indeed, once the gradient estimate \eqref{char_sol_thm} is established, all other path-space estimates follow arguing similarly as in \cite[Section 4]{HaslhoferNaber}.\\

The key for proving Theorem \ref{thm_main_intro} (Ricci limit flows and weak solutions) is a new hitting estimate for the Ricci flow. For exposition sake, let us first discuss this estimate in the context of smooth Ricci flows. To this end, let $(g_t)_{t\in (t_0-2r^2,t_0]}$ be a Ricci flow on a closed $n$-dimensional manifold $M$, and recall that the Nash entropy based at $(p_0,t_0)$ is defined by
\begin{equation}
 \mathcal{N}_{(p_0,t_0)}(\tau) := - \int_M \log K(p_0,t_0;\cdot,t_0-\tau)\, d\nu_{(p_0,t_0);t_0-\tau } -\frac{n}{2}(1+\log(4\pi \tau)).
\end{equation}
Recall also that, given any $\eps>0$, the quantitative singular set is defined by
\begin{equation}
\mathcal{S}_\eps = \{ (p,t) \, : \, r_{\mathrm{Rm}}(p,t)\leq \eps\},
\end{equation}
where $r_{\textrm{Rm}}(p,t)$ is the largest $r$ such that $|\mathrm{Rm}|\leq r^{-2}$ on the backwards parabolic ball $P_{-}(p,t;r)$.

\begin{theorem}[hitting estimate for the Ricci flow]\label{thm_hitting_intro}
For all $Y<\infty$, $\delta>0$, and $r_0\in (0,r/2)$, there exists a constant $C=C(n,Y,\delta,r_0,r)<\infty$, such that if $(g_t)_{t\in (t_0-2r^2,t_0]}$ is a Ricci flow with $\mathcal{N}_{(p_0,t_0)}(r_0^2) \ge -Y$ and $r_{\mathrm{Rm}}(p_0,t_0)\geq r_0$, then Brownian motion $X_\tau$ starting at $(p_0,t_0)$ satisfies
\be\label{eq-hittingprobRF} \mathbb{P}_{(p_0,t_0)}\big[\textrm{$X_\tau$ hits $\mathcal{S}_\eps \cap P^*(p_0,t_0; r)$ for some $\tau \in[0,r^2]$}\big]\le C \eps^{2-\delta}\ee  	
for all $\eps>0$.
\end{theorem}
Heuristically, one can of course easily guess the (almost) quadratic dependence on $\eps$ in light of Bamler's codimension-4 partial regularity result and the intuition that the image of Brownian curves is 2-dimensional. Indeed, hitting estimates in related easier situations go back all the way to the classical work of Kakutani \cite{Kakutani}. A sharp hitting estimate for Brownian motion in Euclidean space has been obtained by Benjamini-Pemantle-Peres \cite{BPP}.  Recently,  in \cite{ChoiHaslhofer1} we generalized the Benjamini-Pemantle-Peres estimate to the setting of spaces with Ricci curvature bounded below.\\

A fundamental new difficulty in the context of Ricci flow, in stark contrast to all prior hitting estimates in the literature, is that the heat kernel only has upper bounds, but no lower bounds. To overcome this, we introduce a novel approach to hitting estimates. Roughly speaking, we compensate for the lack of lower heat kernel bounds by making use of the heat kernel geometry as introduced in \cite{bamler2020entropy}, including in particular the properties of $H_n$-centers and $P^\ast$ parabolic balls (see Section \ref{sec_not_prel} for a discussion of these notions).\\

Our proof of the hitting estimate also carries through in the more general setting of noncollapsed Ricci limit flows. In particular, we obtain:

\begin{corollary}[Brownian motion on Ricci limit flows]\label{cor_bm_lim} 
If $\mathcal{X}$ is a noncollapsed Ricci limit flow, and $x\in\mathcal{R}\subset\mathcal{X}$ is a regular point, then the Wiener measure $\mathbb{P}_x$ of Brownian motion starting at $x$ concentrates on the space of continuous space-time curves that stay entirely in the regular part $\mathcal{R}$. 
\end{corollary}

Using these results, we can then establish the infinite dimensional gradient estimate on path space by adapting the argument from \cite{HaslhoferNaber} to our setting. Specifically, we first consider the $\mathrm{O}_n$-frame bundle $\pi:\mathcal{F}\to \mathcal{R}$ over the regular part. Recalling that this bundle comes with a distribution of horizontal $(n+1)$-planes induced by Hamilton's space-time connection, we can then construct a process $U_\tau\in \mathcal{F}$ as unique horizontal lift of the Brownian motion $X_\tau\in\mathcal{R}$. Thanks to Corollary \ref{cor_bm_lim} (Brownian motion on Ricci limit flows) the process $U_\tau$ does not explode. This enables us to define the stochastic parallel transport map
\begin{equation}
P_\tau:=U_0U_\tau^{-1}: T_{X_\tau} \mathcal{R}_{\mathfrak{t}(x)-\tau}\to T_x \mathcal{R}_{\mathfrak{t}(x)},
\end{equation}
which in turn allows us to define the parallel gradient of any cylinder function $F(X)=f(X_{\tau_1},\ldots, X_{\tau_k})$ by
\begin{equation}\label{def_par_grad_intro}
\nabla^\parallel F(X) = \sum_{i=1}^k P_{\tau_i} \mathrm{grad}^{(i)}_{g_{\mathfrak{t}(x)-\tau_i}} f(X_{\tau_1},\ldots, X_{\tau_k}).
\end{equation}
Another key step is to show that if $v$ is a heat flow, then its gradient at any regular point $x\in \mathcal{R}$ is given by the Feynman-Kac type representation formula
\begin{equation}
\mathrm{grad}_{g_{t}} v (x) = \mathbb{E}_{x} \left[ P_{t-s} \mathrm{grad}_{g_{t-s}} v|_{\mathcal{R}_s} \right],
\end{equation}
where $t=\mathfrak{t}(x)$. To show this, we localize on $\mathcal{X}\setminus\mathcal{S}_\eps$ via a suitable cutoff function, and then take the limit $\eps\to 0$ using Theorem \ref{thm_hitting_intro} (hitting estimate for the Ricci flow). Finally, after this is established, we check that the rest of the argument from \cite{HaslhoferNaber} goes through with minor adaptions.\\

This article is organized as follows. In Section \ref{sec_hitting}, we prove Theorem \ref{thm_hitting_intro} (hitting estimate for the Ricci flow). In Section \ref{sec_weak}, we prove Theorem \ref{thm_main_intro} (Ricci limit flows and weak solutions).

\bigskip

\noindent\textbf{Acknowledgements.}
The second author has been supported by an NSERC Discovery Grant and a Sloan Research Fellowship.\\

\bigskip

\section{Hitting estimate for Ricci flow}\label{sec_hitting}

\subsection{Notation and preliminaries}\label{sec_not_prel}

Let $(g_t)_{t\in I}$ be a Ricci flow on a closed $n$-dimensional manifold $M$. The \emph{heat kernel} $K(p,t;q,s)$, where $p,q\in M$ and $s<t$ in $I$, is defined by
\begin{equation}
(\partial_t - \Delta_{g_t}) K (\cdot,\cdot;q,s)=0,\qquad \lim_{t\searrow s} K(\cdot,t; q,s)=\delta_q.
\end{equation} 
By duality, as a function of the last two variables this solves the conjugate problem
\bea (-\p_s -\Delta_{g_s} +R_{g_s})K(p,t;\cdot,\cdot) =0 ,\qquad \lim_{s\nearrow t} K(p,t;\cdot,s) = \delta_{p} . \eea 
The \emph{conjugate heat kernel measure} is defined by
\bea d\nu_{(p,t);s}(q) = K(p,t;q,s)d\mathrm{Vol}_{g_s}(q),\qquad d\nu_{(p,t);t}=\delta_p   .\eea  
Note that this is a probability measure. We often write
\bea d\nu_{(p,t);s}(q) = (4\pi\tau)^{-n/2} e^{-f_{(p,t)}(q,s)}d\mathrm{Vol}_{g_s}(q) ,\eea  
where $\tau=t-s$. In terms of the potential $f_{(p,t)}$ the \emph{pointed Nash entropy} is given by
\bea \mathcal{N}_{(p,t)}(\tau) = \int f_{(p,t)}(\cdot,t-\tau ) d\nu_{(p,t);t-\tau } -\frac{n}{2}.  \eea
By Perelman's monotonicity formula \cite{Per1}, the function $\tau\mapsto \tau \mathcal{N}_{(p,t)}(\tau)$ is concave. We also recall from \cite[Proposition 5.2]{bamler2020entropy} that $\tau\mapsto \mathcal{N}_{(p,t)}(\tau)$ is nonincreasing, and hence $\mathcal{N}_{(p,t)}\leq 0$, and
\bea \label{eq-nash-monotone}  \frac{d}{d\tau} \mathcal{N}_{(p,t)}(\tau) \geq \min_{q\in M} R(q,t_0-\tau) -\frac{n}{2\tau}.\eea   

\bigskip

Next, we recall the well known fact (see e.g. \cite[Lemma 2.7]{bamler2020entropy}) that under Ricci flow the 1-Wasserstein distance between conjugate heat kernel measures is monotone,
namely
\bea
s\mapsto d_{W_1(g_s)}(\nu_{(p_1,t_1);s},\nu_{(p_2,t_2);s})\quad \textrm{ is nondecreasing}.
\eea
Here, by Kantorovich duality, the 1-Wasserstein distance between probability measures  is given by
\bea d_{W_1(g)}(\mu_1,\mu_2) = \sup \int_M f d\mu_1 -\int_M f d\mu_2,  \eea
where the supremum is taken over all bounded $1$-Lipschitz functions $f:(M,g)\rightarrow \mathbb{R}$.
Motivated by this, Bamler pointed out that instead of considering \emph{conventional parabolic balls}
 \bea P(p_0,t_0;r):= B_{g_{t_0}}(p_0;r)\times [t_0-r^2,t_0+r^2 ],\eea 
it is often more useful to consider so-called \emph{$P^*$ parabolic balls} defined by
	\bea P^*(p_0,t_0;r):= \big\{(p,t)\in M \times[t_0-r^2,t_0+r^2]\,:\, d_{W_1({g_{t_0-r^2}})}(\nu_{(p_0,t_0);t_0-r^2}, \nu_{(p,t);t_0-r^2})<r    \big\}.\eea 
By  \cite[Proposition 9.4]{bamler2020entropy}, $P^*$ parabolic balls satisfy similar containment principles as conventional parabolic balls, in particular:
	\begin{align}\label{eq_p_past_cont}
	(p_1,t_1)\in P^*(p_2,t_2;r)&\quad\Rightarrow\quad P^*(p_2,t_2;r)\subseteq P^*(p_1,t_1;2r),\nonumber\\
	(p_1,t_1)\in P^*(p_2,t_2;r)&\quad\Rightarrow\quad P^*(p_1,t_1,  r')\subseteq P^*(p_2,t_2;r+r').		
	\end{align}
Moreover, by \cite[Theorem 9.8]{bamler2020entropy}, there is some universal $C<\infty$, such that  if $[t_0-2r^2,t_0]\subseteq I$, then for all $t'\in [t_0-r^2, t_0+r^2]$ the volume of the time $t'$-slices is bounded by
\bea \label{vol_of_slices}\mathrm{Vol}_{g_{t'}}(P^*(p_0,t_0;r)\cap \{t=t'\})\le  C e^{\mathcal{N}_{(p_0,t_0)}(r^2)}r^n. \eea 
We will also need the covering result from \cite[Theorem 9.11]{bamler2020entropy}, which says that there is some universal constant $C<\infty$ with the following significance: If $[t_0-2r^2,t_0]\subseteq I$, then for any $X\subseteq P^*(p_0,t_0;r)$ and any $\lambda \in(0,1)$, we can find points $(q_1,s_1)$, $\dots$, $(q_N,s_N)$ in $X$ such that 
\bea\label{cov_lemma} X \subseteq \bigcup_{i=1}^N P^*(q_i,s_i;\lambda r)\quad \text{ and } \quad N\le C \lambda ^{-(n+2)}. \eea 
Now, assuming $[t_0-2r^2,t_0]\subseteq I$ and $\mathcal{N}_{(p_0,t_0)}(r^2)\geq -Y$, if we consider the quantitative singular set
\begin{equation}
\mathcal{S}_\eps = \{ (p,t) \, : \, r_{\mathrm{Rm}}(p,t)\leq \eps\},
\end{equation}
where $r_{\textrm{Rm}}(p,t)$ is the largest $r$ such that $|\mathrm{Rm}|\leq r^{-2}$ on $P_{-}(p,t;r)=B_{g_{t}}(p;r)\times [t_0-r^2,t_0]$, then by Bamler's quantitative parabolic $\ast$-Minkowski codimension-4 bound \cite[Theorem 1.26]{bamler2020structure} we can find points $(q_1,s_1)$, $\ldots$, $(q_N,s_N)$$\in$$ \mathcal{S}_\eps\cap P^\ast_{-}(p_0,t_0;r)$ such that \bea\label{bam_codim4} \mathcal{S}_\eps\cap P^\ast_{-}(p_0,t_0;r) \subseteq \bigcup_{i=1}^N P^*(q_i,s_i;\eps)\quad \text{ and } \quad N\le C \eps ^{-(n-2)-\delta}, \eea 
where $C<\infty$ is a constant that only depends on $n,Y,r$ and $\delta$. Note that for smooth flows we could equally well work with two-sided parabolic balls, but for the generalization to noncollapsed limit flows it is better to use backwards parabolic balls $P^*_{-}(p_0, t_0;r )=P^*(p_0, t_0;r )\cap\{t\leq t_0\}$.\\
Finally, in general there is no containment between $P$ and $P^\ast$ parabolic balls. However, if we assume $r_{\mathrm{Rm}}(p,t)\geq r$ then by \cite[Corollary 9.6]{bamler2020entropy} we have
\begin{equation}\label{p_past_cont}
P_-(p_0,t_0;\eta r)\subseteq P^\ast _-(p_0,t_0;r)\qquad \textrm{and} \qquad P^\ast_-(p_0,t_0;\eta r)\subseteq P _-(p_0,t_0;r),
\end{equation}
where $\eta>0$ is a universal constant.

\bigskip

Next, by an important discovery of Bamler \cite[Corollary 3.7]{bamler2020entropy}, under Ricci flow
\begin{equation}\label{var_mon}
s\mapsto \mathrm{Var}_{g_s}(\nu_{(p_1,t_1);s},\nu_{(p_2,t_2);s}) + H_n s\quad \textrm{ is nondecreasing}.
\end{equation}
Here, $H_n=  \pi^2(n-1) /2+4$, and the variance between two probability measure on $(M,g)$ is defined as 
\bea \mathrm{Var}_g(\mu_1,\mu_2) = \iint_{M\times M}  d_g^2(x_1,x_2) d\mu_1(x_1)d\mu_2(x_2). \eea
Motivated by this, as in \cite[Definition 3.10]{bamler2020entropy} a point $(q,s)$ is called an \emph{$H_n$-center} of $(p,t)$ if $s\le t$ and
\bea  \mathrm{Var}_{g_s} (\delta_q , \nu_{(p,t);s})\le H_n (t-s).\eea
As a direct consequence of \eqref{var_mon}, given any $(p,t)$ and $s\leq t$, there always exists at least one $H_n$-center $(q,s)$ of $(p,t)$ and the distance between any two such $H_n$-centers is bounded by 
\bea
d_{g_s}(q,q') \le 2\sqrt { H_n(t-s)}.
\eea
Moreover, as a direct consequence of the definitions for any $A<\infty$ one has
 \bea\label{concent_inequ} \nu_{(p,t);s}(B_{g_s}(q,\sqrt{AH_n (t-s})) \ge 1-A^{-1}.\eea 
Finally, in general there is no universal bound on the distance from $H_n$-centers to  the base point $p$. However, if we assume for instance $r_{\mathrm{Rm}}(p,t)\geq r$, then by \cite[Proof of Proposition 9.5]{bamler2020entropy} there is universal $C<\infty$, such that for all $H_n$-centers $(q,s)$ with $s\in [t-C^{-1}r^2,t)$ there holds 
\bea \label{eq-distH_n}d_{s} (q, p) \le C \sqrt {t-s} .\eea

\bigskip

To conclude this subsection, let us discuss heat kernel bounds. By  \cite[Theorem 7.2]{bamler2020entropy}, if $R\ge R_{\textrm{min}} $ and $[t-\tau,t] \subseteq I$, then for some $C=C(\tau \cdot R_{\textrm{min}})<\infty$ we have the upper bound
\bea K(p,t; q,t-\tau ) \le \frac{C}{\tau^{n/2}} e^{-\mathcal{N}_{(p,t)}(\tau) } e^{ -\frac{d_{t-\tau}(p_{t-\tau}, q)^2}{10\tau}}, \label{eq-HKupper}\eea 
	where $(p_{t-\tau},t-\tau)$ is any $H_n$-center of $(p,t)$. In general, there are no corresponding lower bounds.

 \bigskip

\subsection{Proof of the hitting estimate}

In this subsection, we prove Theorem \ref{thm_hitting_intro} (hitting estimate for the Ricci flow). By time translation and parabolic rescaling we may assume that $t_0=0$ and $r=1$, i.e. it suffices to prove:

\begin{theorem}[hitting estimate for the Ricci flow; restated]\label{thm_hitting_restated}
For all $Y<\infty$, $\delta>0$, and $r_0\in (0,1/2)$, there exists a constant $C=C(n,Y,\delta,r_0)<\infty$, such that if $(g_t)_{t\in (-2,0]}$ is a Ricci flow with $\mathcal{N}_{(p_0,0)}(r_0^2) \ge -Y$ and $r_{\mathrm{Rm}}(p_0,0)\geq r_0$, then Brownian motion $X_\tau$ starting at $(p_0,0)$ satisfies
\be\label{eq-hittingprobRF_restated} \mathbb{P}_{(p_0,0)}\big[\textrm{$X_\tau$ hits $\mathcal{S}_\eps \cap P^*(p_0,0; 1)$ for some $\tau \in[0,1]$}\big]\le C \eps^{2-\delta}\ee  	
for all $\eps>0$.
\end{theorem}

\begin{proof} To begin with, let us observe that since the flow is defined on the interval $(-2,0]$, the maximum principle for the evolution of scalar curvature under Ricci flow implies
\begin{equation}\label{eq_scalar_lower}
R\geq -n/2 \qquad \textrm{for}\qquad t\in[-1,0].
\end{equation}
Together with \eqref{eq-nash-monotone} and the assumption $\mathcal{N}_{(p_0,0)}(r_0^2) \ge -Y$ this yields
\begin{equation}\label{eq_nash_univlower}
 \mathcal{N}_{(p_0,0)}(1) \ge -C(r_0,Y).
\end{equation}
Hence, we have all the estimates from the previous subsection, which depend on a lower scalar bound and/or a lower entropy bound, at our disposal. In the following, we will simply write $C$ for constants that only depend on $n,Y,\delta$ and $r_0$, and are allowed to change from line to line. Also, we can assume throughout that $\eps\leq r_0/10$, since otherwise there is nothing to prove.\\

As above, denote by $X_\tau$ Brownian motion on our Ricci flow starting at $(p_0,0)$. Given any closed subset $\mathcal{A} \subseteq M \times [-1, 0]$, we consider the hitting time
\begin{equation}
 \tau_{\mathcal{A}} := \inf \{\tau >0\,:\, X_{\tau} \in \mathcal{A} \} \in [0,\infty].
\end{equation}
Note that $\tau_{\mathcal{A}}\wedge 1$ is a stopping time. Let $\mu$ be the distribution of $X_{\tau_{\mathcal{A}}\wedge 1}$, i.e. set
\be  \mu (\mathcal{A}') := \mathbb{P}_{(p_0, 0)} [X_{\tau_{\mathcal{A}}\wedge 1}  \in \mathcal{A}']\ee
for any Borel set ${\mathcal{A}}'\subseteq{\mathcal{A}}$. Observe that
\be\label{property_mu}\mathbb{P}_{(p_0, 0)}[ X_\tau  \in  \mathcal{A} \textrm{ for some } { 0\le \tau \le  1} ]=\mu(\mathcal{A}). \ee  
In the following, we write $\mathcal{A}'_s:=\mathcal{A}'\cap\{ t=s\}$ for the time-slices. Our first goal is to show:
\begin{claim}[hitting distribution]\label{claim_gaussian} The hitting distribution measure $\mu$ satisfies
 \begin{multline}\label{eq-222} \int_{-1}^0\int_{\mathcal{A}'_s} \int _{\mathcal{ A}\cap\{t\ge s\}}  \ K(p,t;q, s)\,d\mu(p,t)  \, d\mathrm{Vol}_{g_{s}}(q)\, ds \\
 \le \int_{-1}^{0}  \int_{\mathcal{A}'_s}\frac{C}{(-s)^{n/2}}  e^{ -\frac{  d_{s}(p_s, q)^2}{10(-s)}}\,  d\mathrm{Vol}_{g_{s}}(q)\, ds,\end{multline}
 where $(p_s,s)$ is any $H_n$-center of $(p_0,0)$.  
\end{claim}

\begin{proof}[{Proof of Claim \ref{claim_gaussian}}]
Consider the expected occupancy time
\bea \label{eq-25}
\mathbb{E}_{(p_0, 0)}\left[ \int_0^{1} 1_{\{X_{\tau}\in \mathcal{A}'\} } d\tau \right] = \int_{-1}^{0} \int_{\mathcal{A}'_s} K(p_0, 0; q, s)\, d\mathrm{Vol}_{g_{s}}(q)\, ds. 
\eea
By the upper heat kernel bound \eqref{eq-HKupper}, remembering also \eqref{eq_scalar_lower} and  \eqref{eq_nash_univlower}, we can estimate
 \bea \label{eq-221}  \mathbb{E}_{(p_0, 0)}\left[ \int_0^{1} 1_{\{X_{\tau}\in \mathcal{A}'\} } d\tau \right]
 &\le \int_{-1}^{0}\int_{\mathcal{A}'_{s}} \frac{C}{(-s)^{n/2}} e^{ -\frac{-d_{s}(
p_s, q)^2}{10(-s)}}\,   d\mathrm{Vol}_{g_{s}}(q)\,ds \,  
 \eea
where $(p_s,s)$ is any $H_n$-center of $(p_0,0)$. On the other hand, we can also compute the expected occupancy time of $\mathcal{A}'$ by conditioning on $X_{\tau_{\mathcal{A}} \wedge  1}$. Specifically, observing that $X_{(\tau_{\mathcal{A}}\wedge 1) +\tau}$ is a Brownian motion with initial distribution $\mu$, and using the strong Markov property, we infer that
\bea\mathbb{E}_{(p_0, 0)}\left[ \int_0^{1} 1_{\{X_{\tau}\in \mathcal{A}'\} } d\tau \right] \ge  \int _{\mathcal{ A}} \int_{-1}^t\int_{\mathcal{A}'_s} K(p,t;q, s)\,  d\mathrm{Vol}_{g_{s}}(q)\, ds \, d\mu(p,t). \eea 
Changing the order of integration, and combining the above inequalities, the claim follows.
 \end{proof} 

\bigskip 

We now fix
\bea
\mathcal{A}:= \{x\in P^*_{-}(p_0, 0;1): \eps/2\leq r_{\mathrm{Rm}}(x)\leq \eps\}.
\eea
Since $r_{\mathrm{Rm}}(p_0,0) \geq 10\e$ at the initial point, and $r_{\mathrm{Rm}}= \e$ on the support of $\mu$, we see that
\bea\label{what_we_compute}\mu (\mathcal{A})=\mathbb{P}_{(p_0,0)}\big[\text{$X_\tau$ hits $\mathcal{S}_\eps  \cap P^*(p_0, 0; 1) $ for some $\tau \in[0,1]$}\big] .\eea 
In the standard proof in the elliptic setting,  see e.g. our prior paper \cite{ChoiHaslhofer1}, the next step would be to estimate the capacity-type integral $\iint_{A\times A} K d\mu d\mu$, which however only works if $A$ is a subset of a fixed space. In our current space-time setting, we consider instead the averaged quantity
\bea  \mathcal{I}:=\label{eq-234-0_av}
\int_{\mathcal{A}} \int_{P^*_{-}(q, s; 4\eta \varepsilon )} \int_{\mathcal{A}\cap\{t\ge s'\}}   \ K(p,t;q' ,s') \,d\mu(p,t) \, d\mathrm{Vol}_{g_{s'}}(q' )  ds' \, d\mu(q, s)  ,\eea where $\eta>0$ is a small constant to be chosen below.
Using Claim \ref{claim_gaussian} (hitting distribution) we can estimate
 \bea  \label{eq-234-0}\mathcal{I}\le\int_{\mathcal{A}} \int_{P^*_{-}(q, s;4\eta \varepsilon )} \frac{C}{(-s')^{n/2}}  e^{ -\frac{  d_{s'}(p_{s'} , q' )^2}{10(-s')}}  d\mathrm{Vol}_{g_{s'}}(q' ) ds' \, d\mu(q, s),\eea
 where $(p_{s'},s')$ is any $H_n$-center of $(p_0,0)$ as above.  To proceed, we observe that if $(q,s)\in\textrm{spt}(\mu)$ and $(q',s')\in P^*_{-}(q, s;4\eta \varepsilon )$, then fixing $\eta=\eta(n)$ small enough we have the bound 
\bea \label{bound_to_show_integrand} \frac{1}{(-s')^{n/2} }e^{-\frac{d_{s'}(p_{s'} ,q')^2}{10(-s')}} \leq  C. \eea
Indeed, for sufficiently small $\eta$, if $-s'\leq \eta r_0^2$ then using in particular \eqref{p_past_cont}  and \eqref {eq-distH_n} we see that $d_{s'}(p_{s'},q')\geq \eta r_0$, and consequently the left hand side of \eqref{bound_to_show_integrand} is bounded by some $C=C(r_0)<\infty$. On the other hand,  if $-s'\geq \eta r_0^2$ then the left hand side is clearly bounded by $(\eta r_0)^{-n/2}$.
Together with the bound  \eqref{vol_of_slices} for the volume of $P^\ast$ parabolic balls, this yields
\bea  \mathcal{I}\le C \varepsilon^{n+2}\mu(\mathcal{A}).\label{I_upper}
\eea 

\bigskip

Next, we would like to bound our quantity $\mathcal{I}$ from below, by estimating the contribution close to the diagonal. Specifically, let us consider $P^*_{i}=P^*(	p_i,t_i; \eta \varepsilon )$ for some $(p_i,t_i) \in \mathcal{A}$. Recall that if $(p,t)\in \mathrm{spt}(\mu)$, then $r_{\mathrm{Rm}}(p,t)=\e $. Together with \eqref{eq-distH_n}, we thus infer that there is some universal $A\in (1,\infty)$ with the following significance: If $(p,t), (q,s) \in P^*_i \cap \mathrm{spt}(\mu)$ satisfy $t\leq s$, then for each  $s'\in  [{t}-A^{-1}(\eta \varepsilon) ^2,t]$ there is  an $H_n$-center  $(p_{s'},s')$ of $(p,t)$ such that
\bea
B_{g_{s'}}(p_{s'},\sqrt{{2}H_n(t-s')})\subseteq P^*_{-}(q, s;4\eta \varepsilon ).
\eea
Combined with \eqref{concent_inequ} this implies
\bea \label{eq-217'}\int_{P^*_{-}(q, s;4\eta \varepsilon )} \ K(p,t;q',s')  \, d\mathrm{Vol}_{g_{s'}}(q') \, ds' \ge \frac{1}{2} \int^t_{ t-A^{-1}(\eta \varepsilon )^2}  ds' \geq C^{-1} \varepsilon^2.\eea 
This yields
\begin{multline} \label{eq-215lowerbound}\int_{P^*_i} \int_{P^*_{-}(q, s;4\eta \varepsilon )}\int_{P^*_i\cap\{t\ge s'\}} \ K(p,t;q',s')  \,d\mu(p,t)\, d\mathrm{Vol}_{g_{s'}}(q') ds' \, d\mu(q, s) \\ 
\geq C^{-1}\e^2 \int _{P^*_i}\int_{P^*_i} 1_{\{t\leq  s\}} d\mu(p,t)d\mu(q,s)  \geq C^{-1}\frac{\eps^2}{2} \mu (P_i^*)^2. \end{multline} 

\bigskip

Now, let $P^*_i=P^*(p_i,t_i; \eta \eps )$, where $(p_i,t_i)\in\mathcal{A}$ for $i=1,\ldots,N$, be a covering of $\mathcal{A}$ with minimal covering number $N=N(\mathcal{A},\eta \eps)$, i.e. 
\bea N = \min\left\{\,n\,:\, \text{ there are } (p_1,t_1),\ldots,(p_n,t_n) \ \in \mathcal{A} \text{ s.t. }  \mathcal{A}\subseteq \bigcup_i P^*(p_i,t_i;\eta \eps ) \right\}. \eea
Observe that, thanks to minimality, the covering multiplicity is uniformly bounded. 
Indeed, if $P^*(p_{i_1},t_{i_1};\eta \eps )$, $\dots$, $P^*(p_{i_{m}},t_{i_m};\eta  \eps )$ from a minimal covering  intersect at some point $(p,t)$, then by the containment relations \eqref{eq_p_past_cont}, these $P^\ast$ parabolic balls are contained in $P^*(p,t; 2\eta \varepsilon )$, and together with the covering result from \eqref{cov_lemma} this implies that $m$ is bounded by some universal constant.  Together with \eqref{eq-215lowerbound} we thus infer that
\bea
\mathcal{I}\geq C^{-1}\eps^2\sum_{i=1}^{N}\mu(P_i^\ast)^2.
\eea
Combined with the elementary inequality
\bea \mu(\mathcal{A})^2 \le \left( \sum _{i=1}^N \mu(P^*_i)\right)^2  \le N  \sum_{i=1}^N  \mu(P^*_i)^2,
\eea  
and the upper bound  from \eqref{I_upper}, this yields
\bea \label{eq-313}\mu(\mathcal{A}) \leq C N \eps^n .\eea
Finally, by Bamler's quantitative parabolic $\ast$-Minkowski codimension-4 bound from \eqref{bam_codim4} we have
\bea
N \leq C \eps^{-(n-2)-\delta},
\eea
and remembering \eqref{what_we_compute} we thus conclude that
\bea
\mathbb{P}_{(p_0,0)}\big[\text{$X_\tau$ hits $\mathcal{S}_\eps  \cap P^*(p_0, 0; 1) $ for some $\tau \in[0,1]$}\big]\leq C \eps^{2-\delta}.
\eea  
This finishes the proof of the theorem.
\end{proof}

\bigskip

\begin{corollary}[occupancy time]\label{cor_occ_time}
Under the same assumption as in Theorem \ref{thm_hitting_restated}, we have 
\bea  \mathbb{E}_{(p_0,0)} \left[ \int_0^1 1_{\{X_\tau \in \mathcal{S}_\eps\cap P^*(p_0, 0;1) \}}d\tau \right] \le C(n,Y,\delta,r_0 ) \e ^{4-\delta }   .\eea

\begin{proof} By definition of Brownian motion it holds that
\bea 
\mathbb{E}_{(p_0, 0)}\left[ \int_0^{1} 1_{\{X_\tau \in \mathcal{S}_\eps \cap P^*(p_0, 0;1) \}} d\tau \right] = \int_{\mathcal{S}_\eps\cap P^*_{-}(p_0, 0;1) } K(p_0, 0; q, s)\, d\mathrm{Vol}_{g_{s}}(q)\, ds. 
\eea
Similarly as in \eqref{bound_to_show_integrand} we have the estimate
\bea
\sup_{(q,s)\in \mathcal{S}_\eps\cap P^*_{-}(p_0, 0;1)} K(p_0, 0; q, s)\leq C.
\eea
Now, by Bamler's quantitative parabolic $\ast$-Minkowski codimension-4 bound from \eqref{bam_codim4} the set  $\mathcal{S}_\eps \cap P^*_{-}(p_0, t_0;1)$ can be covered by $C \e ^{-n+2-\delta }$ number of $P^*$-parabolic balls of radius $\eps$ centered at $(q_i,s_i)\in \mathcal{S}_\eps\cap P^*_{-}(p_0, 0;1)$. Moreover, by \eqref{vol_of_slices}
 the space-time volume of each $P^\ast$ parabolic ball in the covering is bounded by $C \e ^{n+2}$. Combining the above facts yields the assertion.
\end{proof}
 
\end{corollary}

\bigskip

\section{Ricci limit flows and weak solutions}\label{sec_weak}

\subsection{Preliminaries on Ricci limit flows}\label{subsec_prelim_limit_flows}

As in \cite[Definition 3.2]{bamler2020compactness} a \emph{metric flow over $I\subseteq\mathbb{R}$}, 
\begin{equation}
\mathcal{X}=\left(\mathcal{X},\mathfrak{t},(d_t)_{t\in I},(\nu_{x;s})_{x\in \mathcal{X}, s\in I,s\leq \mathfrak{t}(x)}\right),
\end{equation}
consists of a set $\mathcal{X}$, a time-function $\mathfrak{t}:\mathcal{X}\to \mathbb{R}$, complete separable metrics $d_t$ on the time-slices $\mathcal{X}_t=\mathfrak{t}^{-1}(t)$, and probability measures $\nu_{x;s}\in \mathcal{P}(\mathcal{X}_s)$, such that:
\begin{itemize}
\item $\nu_{x;\mathfrak{t}(x)}=\delta_x$ for all $x\in \mathcal{X}$, and for all $t_1\leq t_2\leq t_3$ in $I$ and all $x\in\mathcal{X}_{t_3}$ we have the Kolmogorov consistency condition
\begin{equation}\label{Kolm_cons}
\nu_{x; t_1} = \int_{\mathcal{X}_{t_2}} \nu_{\cdot; t_1}\, d\nu_{x; t_2} .
\end{equation}
\item For all $s<t$ in $I$, any $T>0$, and any $T^{-1/2}$-Lipschitz function $f_s:\mathcal{X}_s\to\mathbb{R}$, setting $v_s=\Phi\circ f_s$, where $\Phi:\mathbb{R}\to (0,1)$ denotes the antiderivative of $(4\pi)^{-1}e^{-x^2/4}$, the function
\begin{equation}\label{grad_est_Bam}
v_t:\mathcal{X}_t\to \mathbb{R},\qquad x \mapsto \int_{\mathcal{X}_s} v_s \, d\nu_{x; s}
\end{equation}
is of the form $v_t=\Phi\circ f_t$ for some $(t-s+T)^{-1/2}$-Lipschitz function $f_t:\mathcal{X}_t\to\mathbb{R}$.
\end{itemize}
In particular, on any metric flow we always have a \emph{heat flow} of integrable functions and a \emph{conjugate heat flow} of probability measures, which are defined for $s\leq \mathfrak{t}(x)$ via the formulas
\begin{equation}\label{heat_flow_def}
v_{\mathfrak{t}(x)}(x):= \int_{\mathcal{X}_s} v_s \, d\nu_{x; s},\qquad
\mu_s:= \int_{\mathcal{X}_t} \nu_{x; s} \, d\mu_{\mathfrak{t}(x)}(x)\, .
\end{equation}
We recall from \cite[Definition 3.30 and Definition 4.25]{bamler2020compactness} that a metric flow $\mathcal{X}$ is called \emph{$H$-concentrated}  if for all $s\leq t$ in $I$ and all $x_1,x_2\in \mathcal{X}_t$ it holds that
\begin{equation}
\textrm{Var}(\nu_{x_1; s}, \nu_{x_2; s})\leq d_t^2(x_1,x_2)+H(t-s),
\end{equation}
and is called \emph{future continuous at $t_0\in I$} if for all conjugate heat flows $(\mu_t)_{t\in I'}$ with finite variance and $t_0\in I'$, the function
$
t\mapsto \int_{\mathcal{X}_t}\int_{\mathcal{X}_t} d_t\, d\mu_t \, d\mu_t
$
is right continuous at $t_0$.\\

As in \cite[Definition 5.1]{bamler2020compactness} a \emph{metric flow pair} over an interval $I$,  consists of a metric flow $\mathcal{X}$ over $I'\subseteq I$ with $|I\setminus I'|=0$, and a conjugate heat flow $(\mu_t)_{t\in I'}$ on $\mathcal{X}$ with $\textrm{spt}(\mu_t)=\mathcal{X}_t$ for all $t\in I'$.\\

Now, any sequence $(M^i,(g^i_t)_{t\in I^i},p^i)$ of pointed Ricci flows on closed $n$-dimensional manifolds, where $I^i=(-T^i,0]$ for ease of notation, can be viewed as sequence of metric flow pairs by considering the associated metric flows $\mathcal{X}^i=M^i \times I^i$
and the conjugate heat flows $(\mu_t^i)=(\nu_{(p^i,0);t})_{t\in I^i}$. By Bamler's compactness theory \cite{bamler2020compactness} after passing to a subsequence we have $\mathbb{F}$-convergence on compact time intervals to a metric flow pair $(\mathcal{X},(\nu_{x_\infty;t})_{t\in (-T_\infty,0]})$, where $\mathcal{X}$ is a future continuous, $H_n$-concentrated metric flow of full support over $(-T_\infty,0]$, and $T_\infty=\lim_{i\to \infty} T^i\in (0,\infty]$.\\

We will assume throughout that the sequence of Ricci flows is \emph{noncollapsed}, namely that there are constants $\tau_0>0$ and $Y_0<\infty$ such that
\begin{equation}
\mathcal{N}_{(p_i,0)}(\tau_0)\geq -Y_0.
\end{equation}
Then, by Bamler's partial regularity theory \cite{bamler2020structure} we have the decomposition
\begin{equation}
\mathcal{X}\setminus\{x_\infty\}=\mathcal{R}\cup \mathcal{S}
\end{equation}
into regular and singular part, where the singular part $\mathcal{S}$ has parabolic $\ast$-Minkowski dimension at most $n-2$. Furthermore, the $\mathbb{F}$-convergence is smooth on the regular part $\mathcal{R}$, and the regular part can be equipped with a unique structure of a Ricci flow space-time,
\begin{equation}
\mathcal{R}=(\mathcal{R},\mathfrak{t},\partial_{\mathfrak{t}},g),
\end{equation}
 as introduced by Kleiner-Lott \cite{KL1}. Hence, $\mathcal{R}$ is a smooth $(n+1)$-manifold, the time-function $\mathfrak{t}:\mathcal{R}\to (-T_\infty,0)$ is smooth without critical points, $\partial_{\mathfrak{t}}$ is a vector field on $\mathcal{R}$ satisfying $\partial_{\mathfrak{t}} {\mathfrak{t}} =1$, and $g=(g_t)_{t\in (-T_\infty,0)}$ is a smooth inner product on $\textrm{ker}(d\mathfrak{t})\subset T\mathcal{R}$ satisfying the Ricci flow equation
 \begin{equation}\label{Ricci_flow_fancy_eq}
 \mathcal{L}_{\partial_{\mathfrak{t}}} g = -2 \textrm{Ric}(g).
 \end{equation}

\bigskip

\subsection{Brownian motion and stochastic parallel transport}\label{sec_BM_stoch}

In this subsection, we explain that every noncollapsed Ricci limit flow can be canonically equipped with a notion of Brownian motion and stochastic parallel transport.
In the following $\mathcal{X}$ denotes any noncollapsed Ricci limit flow, as in the previous subsection. Recall in particular that its regular part $\mathcal{R}\subset \mathcal{X}$ has the structure of a Ricci flow space-time.

\begin{definition}[Brownian motion]\label{def_BM_lim}
\emph{Brownian motion} $\{X_\tau\}_{\tau \in[0,T_\infty-|\mathfrak{t}(x)|)}$ starting at $x\in \mathcal{X}$ is defined by
\begin{equation}\label{def_Brownian}
\mathbb{P}_x[ X_{\tau_1} \in B_1,\ldots,X_{\tau_k} \in B_k] 
=  \int_{B_1\times \ldots\times B_k}d\nu_{x;\mathfrak{t}(x)-\tau_1}(x_1)\ldots d\nu_{x_{k-1};\mathfrak{t}(x)-\tau_k}(x_k),
\end{equation}
for any Borel sets $B_i\subseteq \mathcal{X}_{\mathfrak{t}(x)-\tau_i}$ and any times $0\leq \tau_1<\ldots< \tau_k< T_\infty-|\mathfrak{t}(x)|$.
\end{definition}

Thanks to the Kolmogorov consistency condition \eqref{Kolm_cons}, there  indeed exists a unique such probability measure by the Kolmogorov extension theorem. A priori the probability measure is defined on the infinite product space $\prod_{\tau\in [0,T_\infty-|\mathfrak{t}(x)|) }\mathcal{X}_{\mathfrak{t}(x)-\tau}$, but we will see momentarily that for $x\in\mathcal{R}$ it actually concentrates on the space of continuous space-time curves that stay entirely in the regular part.\\
Note that in the  proof of Theorem \ref{thm_main_intro} (hitting estimate for the Ricci flow) we only used the relation between the Wiener measure and the heat kernel, which now holds true by Definition \ref{def_BM_lim} (Brownian motion), and Bamler's estimates that we recalled in Section \ref{subsec_prelim_limit_flows}, which as explained in \cite{bamler2020compactness,bamler2020structure} hold for limit flows as well. Let us elaborate on a few technical points: The lower scalar bound \eqref{eq_scalar_lower} was only used to derive the Nash entropy bound \eqref{eq_nash_univlower} and to get a uniform constant in the heat kernel upper bound \eqref{eq-HKupper}. In the setting of this subsection, one has instead a lower scalar bound along the sequence of smooth flows, and can then pass the Nash entropy bound and the heat kernel upper bound to the limit flow using the definition of $\mathbb{F}$-convergence and \cite[Theorem 1.11]{bamler2020structure}. Furthermore, recall that we defined $r_{\mathrm{Rm}}$ by taking the supremum over backwards parabolic balls $P_{-}(p,t;r)$, which is slightly more restrictive than the definition of $r_{\mathrm{Rm}}'$ used in \cite[Theorem 1.31]{bamler2020structure}. Hence, \eqref{bam_codim4} indeed holds for noncollapsed limit flows. \\
In particular, for any $x\in\mathcal{R}$ we obtain
\begin{equation}
\mathbb{P}_x\big[ \textrm{$X_{\tau}$ hits $\mathcal{S}$ for some $\tau\in [0,T_\infty-|\mathfrak{t}(x)|)$} \big] = 0.
\end{equation}
Hence, the process stays entirely in $\mathcal{R}$ and can be described in terms of the smooth geometry of $\mathcal{R}$. In particular, almost surely $X_\tau$ is a continuous space-time curve satisfying $\mathfrak{t}(X_\tau)=\mathfrak{t}(x)-\tau$.\\

Our next goal is to construct stochastic parallel transport, by adapting the construction from \cite{HaslhoferNaber} to the setting of Ricci flow space-times.
Let $Y$ be a spatial vector field over $\mathcal{R}$, and let $x\in\mathcal{R}$. The \emph{covariant spatial derivative} in direction $X\in T_x\mathcal{R}_{\mathfrak{t}(x)}$ is defined as
\begin{equation}
\nabla_XY=\nabla_X^{g_{\mathfrak{t}(x)}} Y,
\end{equation}
using the Levi-Civita connection of the metric $g_{\mathfrak{t}(x)}$. Define the \emph{covariant time derivative} by
\begin{equation}
\nabla_{\mathfrak{t}} Y = \partial_{\mathfrak{t}} Y +\tfrac{1}{2} \mathcal{L}_{\partial_{\mathfrak{t}}} g(Y,\cdot)^{\sharp_{g}},
\end{equation}
and observe that with this definition the connection is metric, namely $\tfrac{d}{dt}|Y|^2_{g}=2\langle Y,\nabla_{\mathfrak{t}}Y\rangle$.
Next, consider the $\mathrm{O}_n$-bundle $\pi: \mathcal{F}\to \mathcal{R}$ whose fibres $\mathcal{F}_x$ are given by the orthogonal maps $u:\mathbb{R}^n\to (T_x\mathcal{R}_{\mathfrak{t}(x)} , g_{\mathfrak{t}(x)})$, and where $\mathrm{O}_n$ acts from the right via composition.
For any spatial vector $X\in T_x\mathcal{R}_{\mathfrak{t}(x)}$ its horizontal lift $X^\ast$ is simply given as horizontal lift with respect to Levi-Civita connection of the metric $g_{\mathfrak{t}(x)}$. In particular, we have $n$ canonical horizontal vector fields
\begin{equation}
H_i(u)=(ue_i)^\ast,
\end{equation}
where $u\in \mathcal{F}$, and $e_1,\ldots,e_n$ denotes the standard basis in $\mathbb{R}^n$. Furthermore, denote by $D_{\mathfrak{t}}$ the horizontal lift of the time vector field $\partial_{\mathfrak{t}}$.
Similarly as in \cite[Lemma 3.1 and 3.3]{HaslhoferNaber} covariant derivatives of spatial tensor fields on $\mathcal{R}$ can be expressed in terms of horizontal derivatives of the associated equivariant functions on the frame bundle. For example, identifying spatial vector fields $Y$ on $\mathcal{R}$ with equivariant functions $\tilde{Y}:\mathcal{F}\to \mathbb{R}^n$ via $\tilde{Y}(u)=u^{-1} Y({\pi u})$, we have
\begin{equation}\label{eq_lift_dt}
 \widetilde{\nabla_{\mathfrak{t}} Y} = D_{\mathfrak{t}} \tilde{Y} .
\end{equation}
Now, given any initial frame $u\in \mathcal{F}_x$, there exists a unique horizontal lift $U_\tau$ of $X_\tau$, i.e. a horizontal process $U_\tau$  starting at $U_0=u$ such that $\pi(U_\tau)=X_\tau$. Concretely, using the Eells-Elworthy-Malliavin formalism, similarly as in \cite[Section 3.2]{HaslhoferNaber}, this process is given as the solution of the stochastic differential equation
\begin{equation}\label{SDE}
dU_\tau= -D_{\mathfrak{t}}(U_\tau)d\tau + \sum_{i=1}^n H_i(U_\tau)\circ dW^i_\tau,\qquad U_0=u,
\end{equation}
where $\circ d$ denotes the Stratonovich differential, and we use the normalization
\begin{equation}
dW^i_\tau dW^j_\tau =2\delta_{ij}d\tau.
\end{equation}
Since we have seen above that $X_\tau$ stays entirely in the regular part $\mathcal{R}=\pi(\mathcal{F})$, the solution of \eqref{SDE} does not explode, i.e. we have $U_\tau\in \mathcal{F}$ for all $\tau \in[0,T_\infty-|\mathfrak{t}(x)|)$.

\begin{definition}[stochastic parallel transport]\label{def_stoch_par}
The family of isometries
\begin{equation}
P_\tau:=U_0U_\tau^{-1}: T_{X_\tau} \mathcal{R}_{\mathfrak{t}(x)-\tau}\to T_x \mathcal{R}_{\mathfrak{t}(x)},
\end{equation}
where $U_\tau$ is the horizontal lift of $X_\tau$, is called \emph{stochastic parallel transport}.
\end{definition}

Note that, by equivariance under the $\mathrm{O}_n$-action, $P_\tau$ does not depend on the choice of $u\in \mathcal{F}_x$.

\bigskip

\subsection{Gradient estimate on path space}In this final subsection, we prove that every noncollapsed Ricci limit flow $\mathcal{X}$ is a weak solution in the sense of Haslhofer-Naber.
Recall that a \emph{cylinder function} is a function of the form
\begin{equation}
F(X)=f(X_{\tau_1},\ldots, X_{\tau_k}),
\end{equation}
where  $f: \mathcal{X}_{\mathfrak{t}(x)-\tau_1}\times \ldots\times \mathcal{X}_{\mathfrak{t}(x)-\tau_k}$ is a Lipschitz function with compact support, for some given times $0\leq \tau_1 < \ldots < \tau_k < T_{\infty}-|\mathfrak{t}(x)|$. The \emph{parallel gradient} $\nabla^\parallel F(X)\in T_x \mathcal{R}_{\mathfrak{t}(x)}$ is defined by
\begin{equation}\label{def_par_grad}
\nabla^\parallel F(X) = \sum_{i=1}^k P_{\tau_i} \mathrm{grad}^{(i)}_{g_{\mathfrak{t}(x)-\tau_i}} f(X_{\tau_1},\ldots, X_{\tau_k}),
\end{equation}
where $\mathrm{grad}^{(i)}$ denotes the gradient with respect to the $i$-th entry, and $P_{\tau_i}: T_{X_{\tau_i}} \mathcal{R}_{\mathfrak{t}(x)-\tau_i}\to T_x \mathcal{R}_{\mathfrak{t}(x)}$ denotes stochastic parallel transport (see Definition \ref{def_stoch_par}). The goal of this subsection is to prove:

\begin{theorem}[gradient estimate]
For any $x\in \mathcal{R}$ we have the gradient estimate
\begin{equation}
\left|\mathrm{grad}_{g_{\mathfrak{t}(x)}} \mathbb{E}_x [F] \right| \leq \mathbb{E}_x \left[ |\nabla^\parallel F| \right],
\end{equation}
for all cylinder functions $F$.
In particular, $\mathcal{X}$ is a weak solution of the Ricci flow in the sense of Haslhofer-Naber.
\end{theorem}

\begin{proof}
Suppose first $k=1$. Then, by the definition of Brownian motion from \eqref{def_Brownian} the expectation on the left hand side is given by the heat flow, namely
\begin{equation}
 \mathbb{E}_x [F] = v(x),
 \end{equation}
 where $v$ is the heat flow from \eqref{heat_flow_def} with initial condition $f$ at time $\mathfrak{t}(x)-\tau_1$.
Observe that the gradient of $v$ satisfies
\begin{equation}\label{ev_eq_grad}
\nabla_{\mathfrak{t}} \textrm{grad}_{g} v = \Lap_g \textrm{grad}_{g} v
\end{equation}
on $\mathcal{R}\cap \mathfrak{t}^{-1}( (\mathfrak{t}(x)-\tau_1,\mathfrak{t}(x)])$, by virtue of the Ricci flow equation \eqref{Ricci_flow_fancy_eq}. The key to proceed is the following claim:

\begin{claim}[Feynman-Kac type representation formula]\label{claim_FK}
For any $x\in\mathcal{R}$ we have
\begin{equation}
\mathrm{grad}_{g_{\mathfrak{t}(x)}} v (x) = \mathbb{E}_x \left[ P_{\tau_1} \mathrm{grad}_{g_{\mathfrak{t}(x)-\tau_1}} f\right] .
\end{equation}
\end{claim}

\begin{proof}[Proof of the claim]
Set $Y=\textrm{grad}_{g} v$, and consider the associated equivariant function $\tilde{Y}(u)=u^{-1} Y({\pi u})$. Using \eqref{eq_lift_dt} we see that the lift of the evolution equation  \eqref{ev_eq_grad} is given by
\begin{equation}\label{equat_lifted}
D_{\mathfrak{t}} \tilde{Y} = \Lap_H \tilde{Y},
\end{equation}
where $\Lap_H=\sum_{i=1}^n H_iH_i$ denotes the horizontal Laplacian.\\
Now, for any $\eps>0$, as before denote by $\mathcal{S}_\eps\subseteq \mathcal{X}$ the space-time points with curvature scale less than $\eps$. Let $\eta_\eps:\mathcal{X}\to [0,1]$ be a cutoff function with $\eta_\eps = 1$ on $\mathcal{X}\setminus \mathcal{S}_\eps$ and  $\eta_\eps = 0$ on $\mathcal{S}_{\eps/2}$, and such that
\begin{equation}\label{cut_off_bound}
\eps |\nabla \eta_\eps  | + \eps^2 |\nabla^2 \eta_\eps | + \eps^2 |\partial_{\mathfrak{t}}\eta_\eps |\leq C.
\end{equation}
Set $\tilde{\eta}_\eps:=\eta_\eps\circ \pi$, and consider the truncated function
\begin{equation}
\tilde{Y}^\eps := \tilde{\eta}_\eps \tilde{Y}.
\end{equation}
Similarly as in \cite[Proof of Proposition 3.7]{HaslhoferNaber} the Ito formula on the frame bundle takes the form
\begin{equation}
d\varphi(U_\tau)= \sum_{i=1}^n H_i \varphi(U_\tau) dW_\tau^i - D_{\mathfrak{t}}\varphi(U_\tau) d\tau + \Lap_H \varphi(U_\tau)d\tau.
\end{equation}
Moreover, by the Lipschitz estimate from \eqref{grad_est_Bam} and standard interior estimates we have
\bea
|Y|+\eps |\nabla Y|\leq C.
\eea
Hence, using the equations \eqref{equat_lifted} and \eqref{cut_off_bound} from above, we infer that
\begin{equation}
d\tilde{Y}^\eps(U_\tau)= \mathrm{martingale} + E_\eps\,   d\tau,
\end{equation}
where the error term satisfies
\begin{equation}\label{eq-Eest}
|E_\eps| \leq \frac{C}{\eps^2}\,  1_{\{X_\tau \in \mathcal{S}_\eps\setminus \mathcal{S}_{\eps/2}\}}.
\end{equation}
This implies
\begin{equation}
\left|\tilde{Y}^\eps(u)-\mathbb{E}_u\big[\tilde{Y}^\eps(U_{\tau_1})\big] \right|\leq \frac{C}{\eps^2} \mathbb{E}_x\left[ \int_0^{\tau_1} 1_{\{X_\tau \in \mathcal{S}_\eps\setminus \mathcal{S}_{\eps/2}\}} d\tau \right].
\end{equation}
By Corollary \ref{cor_occ_time} (occupancy time) we have
\begin{equation}
\mathbb{E}_x\left[ \int_0^{\tau_1} 1_{\{X_\tau \in \mathcal{S}_\eps\setminus \mathcal{S}_{\eps/2}\}} d\tau \right]\leq C\eps^{4-\delta}.
\end{equation}
Moreover, using again Theorem \ref{thm_hitting_intro} (hitting estimate for the Ricci flow), and remembering also the Lipschitz estimate from \eqref{grad_est_Bam}, we see that
\begin{equation}
\lim_{\eps\to 0}\mathbb{E}_u\big[\tilde{Y}^\eps(U_{\tau_1})\big]= \mathbb{E}_u\big[\tilde{Y}(U_{\tau_1})\big].
\end{equation}
Also, since $u\in \mathcal{F}_x$, where $x\in\mathcal{R}$, we have 
\begin{equation}
\lim_{\eps\to 0} \tilde{Y}^\eps(u)=\tilde{Y}(u).
\end{equation}
Combining the above fact, we conclude that
\begin{equation}
\tilde{Y}(u)=\mathbb{E}_u\big[\tilde{Y}(U_{\tau_1})\big].
\end{equation}
Pushing down via $\pi$, this establishes the claim.
\end{proof}

Continuing the proof of the theorem, by Claim \ref{claim_FK} (Feynman-Kac type representation formula) and the definition of the parallel gradient from \eqref{def_par_grad} we thus have
\begin{equation}\label{gradient_form_cyl}
\mathrm{grad}_{g_{\mathfrak{t}(x)}} \mathbb{E}_x [F] = \mathbb{E}_x \left[ \nabla^\parallel F \right],
\end{equation}
provided  $F$ is a $1$-point cylinder function. Arguing by by induction on $k$, similarly as in \cite[Proof of Theorem 4.2]{HaslhoferNaber}, where we now use Claim \ref{claim_FK} (Feynman-Kac type representation formula) instead of \cite[Proposition 3.36]{HaslhoferNaber}, we see that the gradient formula \eqref{gradient_form_cyl} holds for $k$-point cylinder functions as well. This implies the assertion of the theorem.
\end{proof}

\bibliography{ChoiHaslhofer.bib}
\bibliographystyle{alpha}

\bigskip

{\sc Beomjun Choi, Department of Mathematics, POSTECH, 77 Cheongam-Ro, Nam-Gu, Pohang, Gyeongbuk, Korea 37673}\\

{\sc Robert Haslhofer, Department of Mathematics, University of Toronto,  40 St George Street, Toronto, ON M5S 2E4, Canada}\\

\end{document}